\documentclass[a4paper, oneside, 10pt]{article} \usepackage[english]{babel}
\usepackage{latexsym,amsfonts,amsmath,enumerate,graphics,enumerate,amsthm}
\title{\textbf{Monotonicity of maximal equicontinuous factors and an application
    to toral flows} } \author{T.~Hauser and T.~J\"ager}

\let\epsilon=\varepsilon

\theoremstyle{definition}
\newtheorem{definition}{Definition}[section]

\theoremstyle{default}

\newtheorem{theorem}[definition]{Theorem}
\newtheorem{proposition}[definition]{Proposition}
\newtheorem{lemma}[definition]{Lemma}
\newtheorem{corollary}[definition]{Corollary}
\newtheorem*{theoremA}{Theorem A}
\newtheorem*{theoremB}{Theorem B}
\newtheorem*{acknowledgement}{Acknowledgement}

\theoremstyle{remark}
\newtheorem{remark}[definition]{Remark}

\let\epsilon=\varepsilon
\let\phi=\varphi
\let\theta=\vartheta
\newcommand{\R}{\mathbb{R}}
\newcommand{\N}{\mathbb{N}}

\newcommand{\Z}{\mathbb{Z}}

\begin{document}
\maketitle

\begin{abstract}
We show that for group actions on locally connected spaces the maximal
equicontinuous factor map is always monotone, that is, the preimages of single
points are connected. As an application, we obtain that if the maximal
continuous factor of a homeomorphism of the two-torus is minimal, then it is
either (i) an irrational translation of the two-torus, (ii) an irrational
rotation on the circle or (iii) the identity on a singleton.
\end{abstract}

\section{Introduction}

The notion of maximal equicontinuous factor (MEF) is a key concept in
topological dynamics and plays an significant role for the understanding,
description and classification of many important system classes, including
Toeplitz flows and other almost automorphic minimal subshifts in symbolic
dynamics \cite{auslander1988minimal,Downarowicz2005ToeplitzFlows} or model sets
and their associated Delone dynamical systems in aperiodic order
(e.g.\ Schlottmann
\cite{Schlottmann1999GeneralizedModelSets,BaakeGrimm2013AperiodicOrder}). It is
also closely related to the notion of (continuous) dynamical eigenfunctions,
since the latter can be viewed as factor maps to a circle rotation and the MEF
comprises all information about such equicontinuous factors.

However, despite its significance, little attention has been given to additional
structural properties that the MEF and its associated factor map may inherit
from the considered dynamical system or the ambient space. One such property is
monotonicity, in the sense that the preimages of singletons under the factor map
are connected subsets of the phase space. This often plays an important role
when it comes to using information about the factor to obtain insight about the
extension. A paradigmatic example in this context is Poincar\'e's classification
of orientation-preserving circle homeomorphisms \cite{poincare:1885}, which
ensures that if the rotation number of such a map is irrational, then it has an
irrational circle rotation as a factor. In this setting, the factor map is known
to be monotone, and this is crucial for deriving further information such as the
unique ergodicity or the uniqueness of the minimal set of the considered circle
homeomorphism (e.g.\ \cite{katok/hasselblatt:1997}).

Considerable efforts have been made over the last decades to extend Poincar\'e's
approach to higher dimensions, and in particular to the two-dimensional torus
\cite{MisiurewiczZiemian1989RotationSets,franks:1988,franks:1989,KoropeckiTal2012StrictlyToral,BoylandDeCarvalhoHall2016NewRotationSets}. In
particular, several conditions in terms of rotation vectors and their
convergence properties have been established that ensure the existence of
irrational circle rotations or irrational torus translations as factors
\cite{Jaeger2009Linearisation,JaegerTal2016IrrationalRotationFactors,JaegerPasseggi2015SemiconjugateToIrrational,%
  JaegerPasseggi2015SemiconjugateToIrrational}. However, in all these situations
it has not been know whether these irrational factors coincide with the
MEF. Likewise, the monotonicity has only been known for factor maps to
irrational circle rotations, where it can be shown {\em `by hand'} (see
\cite{JaegerPasseggi2015SemiconjugateToIrrational}), but not in the case where
the factor is an irrational translation of the torus.\medskip

It turns out that monotonicity of the maximal equicontinuous factor map can be
established in considerable generality.
\begin{theoremA}
  Suppose that $T$ is a topological group acting on a locally connected compact
  Hausdorff space $X$. Then the factor map to the MEF of $(X,T)$ is monotone.
\end{theoremA}
The proof is based on a careful analysis the regionally proximal relation, which
can be used to define the maximal equicontinuous structure relation in a
constructive. The classes of the former are shown to be connected, and this
carries over to the equivalence classes of the latter.\medskip

As an application, we then turn to toral flows and obtain a classification of
the possible minimal equicontinuous factors.

\begin{theoremB}
  Suppose that $f$ is a homeomorphism of the two-torus. If the MEF of
  $(\mathbb{T}^2,f)$ is minimal, then it must be one of the following three:
  \begin{itemize}
  \item[(i)] an irrational translation on the two-torus;
  \item[(ii)] an irrational rotation on the circle;
  \item[(iii)] the identity on a singleton. 
  \end{itemize}
\end{theoremB}

\begin{acknowledgement}
  TJ has been supported by a Heisenberg professorship of the German Research
  Council (DFG grant OE 538/6-1). 
\end{acknowledgement}

\section{Monotonicity of the maximal equicontinuous factor} \label{MonotoneMEF}

\subsection{Notions and preliminaries}

A \emph{flow} is a triple $(X,T,f)$ where $X$ is a topological space, $T$ is a
topological group and $f$ is a continuous map of $X\times T$ to $X$, such that
$f(\cdot,e)$ is the identity on $X$ ($e$ is the identity element of $T$) and
such that $f(f(\cdot,s),t)=f(\cdot,st)$ for $s,t\in T$. Each $t\in T$ defines a
homeomorphism $f^t\colon X \to X$ by $f^t:=f(\cdot,t)$. To simplify notation we
usually write $(X,T)$ instead of $(X,T,f)$ and denote the action of $t\in T$ on
$X$ by $f^t$. A flow $(X,T)$ is called \emph{discrete} if $T=\mathbb{Z}$.

Let $X,Y$ be sets and $R\subseteq X \times Y$. For $x \in X$ and $y \in Y$
define
\[[x]R:=\{y' \in Y;\, (x,y')\in R\} \text{ and } R[y]:=\{x' \in X;\, (x',y)\in R\}.\]
For $M\subseteq X$ and $N\subseteq Y$ define 
\[[M]R:=\bigcup_{x \in M}[x]R \text{ and } R[N]:=R(N):=\bigcup_{y \in N}R[y].\] 
For $S\subseteq Y\times Z$ define 
\[R S:= S \circ R:=\{(x,z)\in X \times Z;\, \exists y \in Y\colon (x,y)\in R \text{ and } (y,z) \in S)\}.\]
If $X=Y$ define $R^{-1}:=\{(y,x);\, (x,y)\in R\}$.

Note that $(RS)Q=R(SQ)$, $([x]S)Q=[x](SQ)$ and $(RS)[y]=R(S[y])$ for sets
$W,X,Y,Z$, elements $x\in X$, $y\in Y$ and relations $R\subseteq W\times X$,
$S\subseteq X\times Y$ and $Q\subseteq Y\times Z$. Thus we usually omit the
brackets and write shortly $RSQ$, $[x]SQ$, and $RS[y]$ respectively.

During this paper we will identify a function $f\colon X \to Y$ with its graph
$\{(f(x),x);\, x\in X\}$. Note that this graph is considered to be a subset of
$Y\times X$. Under this identification we obtain the usual notation
\[f(M)=\bigcup_{x \in M}\{y \in Y;\, (y,x)\in f\}=\{f(x);\, x \in M\}\]
for all $M\subseteq X$ and 
\[f^{-1}(N)=[N]f=\bigcup_{y \in N}\{x \in X;\, (y,x)\in f\}=\{x \in X;\, f(x) \in N\}\]
for all $N \subseteq Y$.

For a metric space $(X,d)$ denote $B_\epsilon:=\{(x,y)\in X^2;\,
d(x,y)<\epsilon\}$ and $\overline{B}_\epsilon=\{(x,y)\in X^2;\,
d(x,y)\leq\epsilon\}$ for $\epsilon>0$. With the above notation we have that
$B_\epsilon[x]$ is the open ball around $x$ with radius $\epsilon$ and
$B_\epsilon(M)=\bigcup\{B_\epsilon[m];\, m \in M\}$ for $M\subseteq X$.

	Let $X$ be a compact Hausdorff space. Then the collection
        $\mathcal{U}_X$ of all neighbourhoods of the diagonal $\Delta\subseteq X
        \times X$ is a uniformity which generates the topology of
        $X$. Furthermore $\mathcal{U}_X$ is the only uniformity that is
        compatible with the topological structure of $X$. 
	
If $S$ is an equivalence relation on a set $X$, denote by $X\big/ S$ the
corresponding partition by equivalence classes. If $S$ is an equivalence
relation on a topological space $X$, then the \emph{quotient space} $X\big/ S$
is equipped with the \emph{quotient topology}, i.e. the finest topology on
$X\big/ S$ such that the \emph{factor map} $\pi_S\colon X \rightarrow
\mathcal{P}: x \mapsto S[x]$ is continuous.  The following lemma can be found in
\cite[Proposition 1]{Daverman}.

\begin{lemma} \label{lem:Uppersemicontinuityproperty_characterization}
	Let $X$ be a topological space and $S$ an equivalence relation on $X$
        such that all elements of $X\big/S$ are closed. The following statements
        are equivalent.
	\begin{itemize}
		\item[(i)] For each $M \in X\big/S$ and each open $U \subseteq
                  X$ with $M \subseteq U$ there exists an open set $V\subseteq
                  X$ with $M \subseteq V$ and such that whenever $N\in X\big/S$
                  satisfies $N\cap V\neq \emptyset$ we have $N \subseteq U$.
		\item[(ii)] For every open subset $U\subseteq X$ the set
                  $U^*:=\bigcup\{M \in X\big/S;\, M \subseteq U \}$ is open.
		\item[(iii)] The projection mapping $\pi_S$ is a closed mapping,
                  i.e. $\pi_S$ maps closed sets to closed sets.
	\end{itemize}
\end{lemma}


An equivalence relation $S$ is called \emph{upper semicontinuous}, if it
satisfies one of the properties in Lemma
\ref{lem:Uppersemicontinuityproperty_characterization} and if every $M \in X
\big/ S$ is a compact subset of $X$. A binary relation $R$ on a topological
space $X$ is called \emph{closed}, if it is closed as a subset of $X \times
X$. The next Proposition shows that on compact Hausdorff spaces the notions of
closed equivalence relations and upper semi-continuous equivalence relations
coincide. We include a proof for the convenience of the reader.

\begin{proposition} \label{pro:Uppersemicontinu_is_closed_CompactHausdorff}
	Let $X$ be a compact Hausdorff space and $S$ an equivalence relation on
        $X$. The following statements are equivalent.
\begin{itemize}
	\item[(i)] $S$ is a upper semicontinuous.
	\item[(ii)] $S$ is closed.
	\item[(iii)] $X\big/S$ is a compact Hausdorff space.
\end{itemize} 
\end{proposition}

\begin{proof}
	The equivalence of (ii) and (iii) is shown in in \cite[Theorem
          8.2.]{Rotman}.  To obtain that (i) implies (iii) let $S$ be upper
        semi-continuous. Let $M,N \in X\big/ S$. Since $X$ is Hausdorff and $M$
        resp.\ $N$ are compact in $X$, these sets are also closed in $X$. Since
        compact Hausdorff spaces are normal\footnote{A topological space $X$ is
          called \emph{normal}, if every two disjoint closed sets of $X$ have
          disjoint open neighborhoods.} there are disjoint open sets $U$ and $V$
        in $X$ such that $M \subseteq U$ and $N \subseteq V$. Let
        $\tilde{U}:=\{M'\in X\big/ S;\, M' \subseteq U\}$ and define $\tilde{V}$
        analogously. By property (ii) of Lemma
        \ref{lem:Uppersemicontinuityproperty_characterization} the sets
        $\pi_S^{-1}(\tilde{U}) =U^*$ and $\pi_S^{-1}(\tilde{V})=V^*$ are open in
        $X$. Thus $\tilde{U}$ and $\tilde{V}$ are open in $X\big/S$. Since
        $\tilde{U}$ and $\tilde{V}$ are disjoint neighbourhoods of $M$
        resp.\ $N$ we have proven $X\big/S$ to be Hausdorff.
	
	To show the implication from (iii) to (i) assume $X\big/S$ to be a
        compact Hausdorff space, i.e.\ $S$ to be closed, and note that by Lemma
        \ref{lem:Uppersemicontinuityproperty_characterization} it is sufficient
        to show that every $M \in X\big/ S$ is compact and that $\pi_S$ is a
        closed map. Every $M\in X\big/S$ is of the form $S[x]=\pi_1(S\cap (X
        \times \{x\}))$ for some $x \in M$ and therefore closed. Here $\pi_1$
        denotes the projection to the first coordinate. Since $X$ is compact we
        have shown $X\big/S$ to consist of compact sets. Further $\pi_S$ is a
        closed mapping as a continuous function from a compact Hausdorff space
        to a Hausdorff space.
\end{proof}
	
If $X$ and $Y$ are topological spaces, we call $R\subseteq X \times Y$
\emph{monotone}, if $[x]R$ is 
particular we call a function $f: X \rightarrow Y$ to be monotone if
$f^{-1}(\{y\})=[y]f$ is connected for every $y \in Y$. Furthermore an
equivalence relation $S$ on a topological space $X$ is monotone if and only if
$\pi_S$ is monotone, since $\pi_S^{-1}(\{x\})=S[x]=[x]S$ for every $x \in X$.

If $S$ is an equivalence relation on a topological space $X$, then the
collection of all $M\subseteq X$ such that $M$ is a connected component of some
$N \in X\big/S$ is a partition of $X$ into connected sets. We call the
corresponding equivalence relation $S_{mon}$ the \emph{monotone refinement
  relation} of $S$. In \cite[I.4.]{Daverman} it is shown, that the monotone
refinement relation $S_{mon}$ of an an upper semicontinuous equivalence relation
$S$ on a Hausdorff space $X$ is upper semicontinuous. Proposition
\ref{pro:Uppersemicontinu_is_closed_CompactHausdorff} therefore implies the
following.

\begin{proposition} \label{pro:monotonization_is_closed}
	If $X$ is a compact Hausdorff space and $S$ is a closed equivalence
        relation on $X$, then the monotone refinement relation of $S$ is also
        closed.
\end{proposition}
	
If $(X,T)$ is a flow, a binary relation $R$ on $X$ is called
\emph{$T$-invariant}, if for all $x \in X$ and all $t \in T$ we have
$f^t([x]R)\subseteq [f^t(x)]R$. Note that this is the case if and only if
$f^t([x]R)= [f^t(x)]R$ for all $x \in X$ and all $t \in T$.

If $\pi\colon X \to Y$ is a factor map between the flows $(X,T)$ and
$(Y,T)$. Then $\{\pi^{-1}(\{y\});\, y \in Y\}$ is a partition of $X$ and the
corresponding equivalence relation $S$ is closed and $T$-invariant. Furthermore
$X\big/S$ and $Y$ are isomorphic. If $S$ is a closed and $T$-invariant
equivalence relation, then $(X\big/S,T)$ is a factor of $(X,T)$. We will
therefore write factors of a flow $(X,T)$ as $(X\big/S,T)$, where $S$ is a
closed and $T$-invariant equivalence relation.

\subsection{Monotonicity of maximal equicontinuous factors} 


The \emph{regionally proximal relation} is defined as $S_{rp}=\bigcap_{\alpha
  \in \mathcal{U}_X}\overline{\alpha T}$, where $\alpha T:=\bigcup_{t\in T}f^t
\alpha f^{(t^{-1})}$ (see \cite{auslander1988minimal}).

\begin{theorem}[{\cite[Chapter 9, Theorems 1 and 3]{auslander1988minimal}}]
	If $(X,T)$ is a flow on a compact Hausdorff space, there is a smallest
        closed $T$-invariant equivalence relation $S_{eq}$ such that the
        quotient flow $(X\big/{S_{eq}},T)$ is equicontinuous. Furthermore
        $S_{eq}$ is the smallest closed and $T$- invariant equivalence relation
        which contains the regionally proximal relation $S_{rp}$.
\end{theorem}

We call $S_{eq}$ the \emph{equicontinuous structure relation} of the flow
$(X,T)$. Furthermore the factor flow $(X\big/{S_{eq}},T)$ is called the
\emph{maximal equicontinuous factor} of $(X,T)$. 

\begin{remark}
For a large class of minimal flows (which include those for which the acting
group is abelian, as well as point distal flows) the regionally proximal
relation $S_{rp}$ is an equivalence relation. Since it is always closed
invariant, it coincides with the equicontinuous structure relation $S_{eq}$ in
these cases \cite[p.141]{auslander1988minimal}.

\end{remark}
In this section will show that
$S_{rp}$ and $S_{eq}$ are monotone, whenever $(X,T)$ is a flow on a compact and
locally connected Hausdorff space. Hence, the respective factor map
$\pi_{S_{eq}}$ is monotone as well. As $S_{rp}$ is reflexive the following
Proposition shows that it is sufficient to prove that $S_{rp}$ is monotone.

\begin{proposition} \label{pro:generatedTinvclosedEqrelisalsoconnected}
	Let $(X,T)$ be a flow on a compact Hausdorff space and $S$ a reflexive
        and monotone binary relation on $X$. Let $\hat{S}$ be the smallest
        closed $T$-invariant equivalence relation that contains $S$. Then
        $\hat{S}$ is monotone.
\end{proposition}

\begin{proof}
	We will show that $\hat{S}$ equals its monotone refinement relation
        $\hat{S}_{mon}$. Clearly $\hat{S}_{mon}\subseteq \hat{S}$. Since
        $\hat{S}$ is the intersection over all closed and $T$-invariant
        equivalence relations $R$ such that $S \subseteq R$ it is left to show
        that $\hat{S}_{mon}$ is closed, $T$-invariant and $S \subseteq
        \hat{S}_{mon}$.

Since $\hat{S}$ is a closed equivalence relation, Proposition
\ref{pro:monotonization_is_closed} implies that $\hat{S}_{mon}$ is closed. In
order to show $\hat{S}_{mon}$ to be $T$-invariant, let $x \in X$ and $t \in
T$. Since $f^t\colon\hat{S}[x] \to f^t(\hat{S}[x])$ is a homeomorphism and
$\hat{S}$ is $T$-invariant we know that $f^t(\hat{S}_{mon}[x])$ is a connected
component of $f^t(\hat{S}[x])=\hat{S}[f^t(x)]$. Note that
$\hat{S}_{mon}[f^t(x)]$ is the connected component of $\hat{S}[f^t(x)]$ that
contains $f^t(x)$. Thus $f^t(\hat{S}_{mon}[x])=\hat{S}_{mon}[f^t(x)]$ and we
have shown $\hat{S}_{mon}$ to be $T$-invariant.

To show $S \subseteq \hat{S}_{mon}$ let $x \in X$. The reflexivity of $S$
implies $x \in S[x]\subseteq \hat{S}[x]$. Since $\hat{S}_{mon}[x]$ is the
connected component of $\hat{S}[x]$ that contains $x$ and $S[x]$ is a connected
subset of $\hat{S}[x]$ that contains $x$, it follows that $S[x]\subseteq
\hat{S}_{mon}[x]$.
\end{proof}

We now split up the proof of the fact that the regionally proximal relation
$S_{rp}$ of a flow $(X,T)$ on a locally connected and compact Hausdorff space is
monotone into several lemmas.

\begin{lemma} \label{lem:monotonerelationsI}
	Let $X$ be a topological space and $R\subseteq X\times X$ reflexive and monotone.
\begin{itemize}
	\item[(i)] If $M$ is a nonempty and connected subset of $X$, then $[M]R$ is connected.
	\item[(ii)] If $f\colon X\to X$ is a homeomorphism, then $f R f^{-1}$ is monotone. 	
	\item[(iii)] If $\left(R_\iota\right)_{\iota \in I}$ is a family of
          reflexive and monotone binary relations on $X$, then $\bigcup_{\iota
            \in I}R_\iota$ is monotone.
\end{itemize}	 
\end{lemma}

\begin{proof}
	To show (i) let $x \in M$. The set $[x]R\cup M$ is a union of connected
        sets with nonempty intersection. Furthermore for each $x \in M$ we have
        $\emptyset \neq M \subseteq [x]R\cup M$. Thus $[M]R=\bigcup_{x \in
          M}[x]R =\bigcup_{x \in M}\left([x]R\cup M\right)$ is connected as an
        union of connected sets with nonempty intersection.

	To show (ii) let $x\in X$ and note that $[x]f=f^{-1}(\{x\})$ is
        connected as the continuous image of a connected set. Thus (i) implies
        $[x](f R)=[[x]f]R$ to be connected. As $f$ is continuous we obtain the
        image $[x](f R f^{-1})=([x](fR))f^{-1}=f([x](f R))$ to be connected.
	
	For $x\in X$ and $\iota \in I$ we obtain $x\in [x]R_\iota$ and that
        $[x]R_\iota$ is connected. Thus $[x](\bigcup_{\iota\in
          I}R_\iota)=\bigcup_{\iota\in I}([x]R_\iota)$ is connected as a union
        of connected sets with non empty intersection. This proves (iii).
\end{proof}

\begin{lemma} \label{lem:connectedasbinrelationentourages}
	Let $X$ be a locally connected and compact Hausdorff space and
        $\mathcal{U}_X$ the unique uniformity on $X$. Then every $\alpha \in
        \mathcal{U}_X$ contains a monotone $\beta\in \mathcal{U}_X$.
\end{lemma}

\begin{proof}
Let $\alpha \in \mathcal{U}_X$. Note that $\mathcal{U}_X$ is the set of all
neighbourhoods of the diagonal $\Delta\subseteq X \times X$. Thus for $x \in X$
there is an open neighbourhood $U_x$ of $x$ with $U_x \times U_x \subseteq
\alpha$. Since $X$ is locally connected there is a connected neighbourhood $V_x$
of $x$ such that $V_x\subseteq U_x$. Let $\beta:=\bigcup_{x \in X}V_x\times
V_x$. Clearly $\beta\subseteq \alpha$ and $\beta$ is a neighbourhood of
$\Delta$, i.e. $\beta \in \mathcal{U}_x$. To show that $\beta$ is monotone let
$x \in X$ and note that
\[[x]\beta=\{z \in X; \, \exists y \in X: x,z \in V_y\}
	= \bigcup \left\{V_y;\, y \in X \text{ and } x \in V_y\right\}
\]
is a union of connected subsets of $X$ whose intersection is nonempty since it
contains $x$.
\end{proof}

The following lemma is a straightforward generalisation of the Cantor intersection theorem for countable intersections.



\begin{lemma} \label{lem:connectedintersectionofdecreasingnet}
	Let $X$ be a compact Hausdorff space and 
	$(X_\alpha)_{\alpha \in A}$ be a decreasing 
	net of nonempty and closed subsets of $X$ such that
	for each $\alpha \in A$ there is $\beta \in A$ with $\beta \geq \alpha$ and such that $X_\beta$ is connected. 
	Then $\bigcap_{\alpha \in A} X_\alpha$ is nonempty and connected.
\end{lemma}


Applying the previous Lemma, we obtain the following results about the closure
and the intersection of monotone equivalence relations.

\begin{lemma} \label{lem:monotonerelationsII}
	Let $X$ be a compact Hausdorff space.
\begin{itemize}
	\item[(i)] If $X$ is locally connected and $R$ is a monotone and
          reflexive relation on $X$, then the closure $\overline{R}$ of $R$ in
          $X\times X$ is also monotone.
	\item[(ii)] If $\left(R_\alpha\right)_{\alpha \in A}$ is a net of closed
          and reflexive binary relations on $X$ such that for all $\alpha \in A$
          there is $\beta \in A$ with $\alpha \leq \beta$ and such that
          $R_\beta$ is monotone, then $\bigcap_{\alpha \in A}R_\alpha$ is
          monotone.
\end{itemize} 
\end{lemma}

\begin{proof}
	Note that $\overline{R} = \bigcap_{(\alpha,\beta)\in \mathcal{U}_X^2}
        \alpha R\beta$. For $x\in X$ we obtain \[[x]\overline{R} =
        \bigcap_{(\alpha,\beta)\in \mathcal{U}_X^2} ([x]\alpha R\beta) =
        \bigcap_{\alpha \in \mathcal{U}_X} \overline{[x]\alpha R}.\] Set
        $X_\alpha=\overline{[x]\alpha R}$ and order $\mathcal{U}_X$ by reversed
        set inclusion to obtain a decreasing net $(X_\alpha)_{\alpha\in
          \mathcal{U}_X}$ of nonempty and closed subsets of $X$. For $\alpha \in
        \mathcal{U}_X$ there is a monotone $\gamma\in \mathcal{U}_X$ such that
        $\gamma\subseteq \alpha$ by Lemma
        \ref{lem:connectedasbinrelationentourages}. Thus $[x]\gamma$ is
        connected and and we find $X_\gamma=\overline{[x]\gamma
          R}=\overline{[[x]\gamma]R}$ to be connected by Lemma
        \ref{lem:monotonerelationsI}(i). Lemma
        \ref{lem:connectedintersectionofdecreasingnet} therefore applies to
        $(X_\alpha)_{\alpha\in \mathcal{U}_X}$ and we obtain that
        $[x]\overline{R}=\bigcap_{\alpha \in \mathcal{U}_X}X_\alpha$ is
        connected. This shows (i).
	
	To see (ii) let $x\in X$ and note that $[x](\bigcap_{\alpha \in
          A}R_\alpha)=\bigcap_{\alpha \in A}([x]R_\alpha)$. Set
        $X_\alpha:=[x]R_\alpha$ for $\alpha \in A$. This defines a decreasing
        net $(X_\alpha)_{\alpha \in A}$ of non empty and closed subsets of
        $X$. For $\alpha \in A$ there is $\beta \in A$ with $\alpha \leq \beta$
        and such that $R_\beta$ is monotone. Thus $X_\beta=[x]R_\beta$ is
        connected and Lemma \ref{lem:connectedintersectionofdecreasingnet}
        yields $[x](\bigcap_{\alpha \in A}R_\alpha)=\bigcap_{\alpha \in
          A}X_\alpha$ to be connected.
\end{proof}

\begin{proposition} \label{pro:regionallyproxRelisconnected}
	Let $(X,T)$ be a flow on a compact and locally connected Hausdorff
        space. Then the regionally proximal relation $S_{rp}$ is monotone.
\end{proposition}

\begin{proof}
	Note that $S_{rp}=\bigcap_{\alpha \in \mathcal{U}_X}\overline{\alpha
          T}$. Denote $R_\alpha=\overline{\alpha T}=\overline{\bigcup_{t\in
            T}f^t \alpha f^{(t^{-1})}}$ for $\alpha \in \mathcal{U}_X$ and order
        $\mathcal{U}_X$ by reversed inclusion to obtain a decreasing net
        $(R_\alpha)_{\alpha \in \mathcal{U}_X}$ of closed and reflexive binary
        relations on $X$. For $\alpha \in \mathcal{U}_X$ there is a monotone
        $\beta \in \mathcal{U}_X$ such that $\beta \subseteq \alpha$.  For $t\in
        T$ we know $f^t$ to be a homeomorphism. Lemma
        \ref{lem:monotonerelationsI}(ii) therefore implies $f^t \beta
        f^{(t^{-1})}$ to be monotone. Note that $f^t\beta f^{(t^{-1})}$ is
        reflexive. Lemma \ref{lem:monotonerelationsI}(iii) and Lemma
        \ref{lem:monotonerelationsII}(i) therefore yield
        $R_\beta=\overline{\bigcup_{t\in T}f^t \beta f^{(t^{-1})}}$ to be
        monotone. We have shown that for $\alpha \in \mathcal{U}_X$ there is
        $\beta \in \mathcal{U}_X$ such that $\beta \subseteq \alpha$ and such
        that $R_\beta$ is monotone. Thus Lemma \ref{lem:monotonerelationsII}(ii)
        can be applied and shows $S_{rp}$ to be monotone.
\end{proof}

Combining Proposition \ref{pro:generatedTinvclosedEqrelisalsoconnected} and	Proposition \ref{pro:regionallyproxRelisconnected} we obtain the following Theorem.  

\begin{theorem} \label{thm:SeqoncompactandlocallyconnectedHSpaceisMonotone}
	The maximal equicontinuous factor $(X\big/S_{eq},T)$ of a flow $(X,T)$ on a compact and locally connected Hausdorff space has a monotone factor map, i.e.\ the equicontinuous structure relation $S_{eq}$ is monotone.
\end{theorem}

\section{On monotone factors of flows on the torus}

Our main application of
Theorem~\ref{thm:SeqoncompactandlocallyconnectedHSpaceisMonotone} is the following. 

\begin{theorem} \label{the:MinimalMEFisT2S1orpt}
If the maximal equicontinuous factor of a homeomorphism of $\mathbb{T}^2$ is
minimal, then it is either conjugate to a flow on $\mathbb{T}^2$, conjugate to
an irrational circle rotation, or conjugate to a flow on a single point.
\end{theorem}

\subsection{Notions and preliminaries}

A flow $(X,T)$ is called \emph{effective}, if for all $t\in T$ not being the
identity element the mapping $f^t$ is not the identity on $X$. A flow $(X,T)$ is
said to be \emph{strongly effective}, if for all $t\in T$ not being the identity
element the mapping $f^t$ is fix point free. Note that every strongly effective
flow is effective and that the reverse is valid for minimal flows, as shown in
\cite[Chapter 1]{auslander1988minimal}.

If $(X,T)$ is a flow, then the set $F$ of all $t\in T$ such that $f^t$ is the identity on $X$ is a closed and normal subgroup. Thus $G:=T\big/F$ defines an action on $X$ by $f^{[t]}(x):=f^t(x)$ for $[t]\in G$ and $x\in X$. The resulting flow $(X,G)$ is effective and $G$ is called the \emph{effective subgroup}. 

If $(X,T)$ is a flow and $(Y,T)$ is a factor with the factor map $\pi\colon X
\to Y$, then $(Y,T)$ is said to be \emph{relatively effective} with respect to
$(X,T)$, if for all $t\in T$ we have that $f^t$ is the identity on $X$, whenever
$f^t$ is the identity on $Y$. Note that $(Y,T)$ is relatively effective with
respect to $(X,T)$ if and only if the effective subgroups of $(X,T)$ and $(Y,T)$
coincide, i.e.\ $t\in T$ acts as the identity on $X$, whenever $t$ acts as the
identity on $Y$. Thus if $(Y,T)$ is relatively effective with respect to $(X,T)$
these flows can be considered to be effective simultaneously, by going over to
the effective subgroup of $T$ if necessary. Note furthermore, that every
effective factor of a flow is relatively effective.

\begin{lemma} \label{lem:factorminimalandeffecitveimpliesextensiontobestronglyeffective}
	If $(X\big/S,T)$ is a strongly effective factor of a flow $(X,T)$, then $(X,T)$ is strongly effective. 
\end{lemma}

\begin{proof}
	If $t\in T$ and $x$ is a fix point of $f^t$ in $X$, then  $S[x]$ is a fix point of $f^t$ in $Y$. 
\end{proof}

During this section denote by $\mathbb{T}^2=\R^2/\mathbb{Z}^2$ the two-torus and by $p\colon
\mathbb{R}^2\to \mathbb{T}^2$ the canonical projection. A subset $M\subseteq
\mathbb{T}^2$ is called \emph{bounded}, if every connected component of $p^{-1}(M)$ is
bounded. Furthermore a subset $M\subseteq \mathbb{T}^2$ is called
\emph{unbounded}, if $M$ is not bounded. An open subset $U\subseteq
\mathbb{T}^2$ is called \emph{doubly-essential}, if $U$ contains two
homotopically non trivial curves which are not homotopic to each other. An open
subset $U\subseteq \mathbb{T}^2$ is called \emph{inessential}, if all curves in
$U$ are homotopically trivial. Furthermore an open subset $U\subseteq
\mathbb{T}^2$ is called \emph{essential}, if $U$ is not doubly essential and not
inessential. Note that an open subset $U\subseteq\mathbb{T}^2$ is essential, if
and only if $U$ contains a homotopically non trivial curve $\gamma$ and every
other curve is either homotopically trivial or homotopic to $\gamma$. A closed
subset $M\subseteq \mathbb{T}^2$ is called \emph{inessential}, if
$\mathbb{T}^2\setminus M$ is doubly essential; \emph{essential}, if
$\mathbb{T}^2\setminus M$ is essential; and \emph{doubly essential}, if
$\mathbb{T}^2\setminus M$ is inessential. It is well-known that if $M$ is a
continuum\footnote{A \emph{continuum} is a compact and connected topological
  space.}, then $M$ is bounded if and only if $M$ is inessential.

A subset $M\subseteq \mathbb{T}^2$ is called a \emph{topological disc}, if it is
homeomorphic to the open disc $D^\circ$, and a \emph{topological annulus}, if it
is homeomorphic to the open annulus $\mathbb{R}\times S^1$. Furthermore a
continuum $M\subseteq \mathbb{T}^2$ is called an \emph{annular continuum}, if
its complement $\mathbb{T}^2\setminus M$ is a topological annulus. It is a well
known fact that a continuum $M$ is an annular continuum if and only if $M$ is
essential and non-separating\footnote{A subset $M$ of a topological space $X$ is
  caled \emph{non-separating}, if $X\setminus M$ has at most one connected
  component.} and that every open and simply connected subset $U\subseteq
\mathbb{T}^2$ is a topological disc (as a consequence of the Riemann Mapping
Theorem). The following version of the classical Moore Theorem can be found in
\cite[Section 3]{TobiasClassofminimalsetsoftorusHomos}. See also
\cite{Daverman}.

\begin{theorem}(Moore)\label{the:Moorestheorem}
	If $S$ an upper semi-continuous and monotone equivalence relation on the
        torus $\mathbb{T}^2$ such that each $M\in X\big/S$ is non separating and
        inessential, then $\mathbb{T}^2\big/S$ is homeomorphic to
        $\mathbb{T}^2$.
\end{theorem} 

If we assume $S$ to be an upper semi-continuous and monotone equivalence
relation, we obtain $X\big/S$ to be a compact and connected metric space. If
each $M\in \mathbb{T}^2\big/ S$ is non separating and essential, i.e.\ an
annular continuum, then the connected components of $\mathbb{T}^2\setminus (M\cup N)$ are
two topological annuli, whenever $M,N\in \mathbb{T}^2\big/S$ with $M\neq
N$. Thus $(\mathbb{T}^2\big/S)\setminus\{M,N\}$ has two connected components. In
\cite[Theorem 2-28]{YoungsTopology} it is shown that a compact and connected
metric space $X$ is homeomorphic to $S^1$ if and only if for all points $x,y\in
X$ with $x\neq y$ the complement $X\setminus \{x,y\}$ has exactly two connected
components. This proves the following proposition.

\begin{proposition}\label{pro:kleinerMoore}
	If $S$ an upper semi-continuous and monotone equivalence relation on the
        Torus $\mathbb{T}^2$ such that each $M\in X\big/S$ is non separating and
        essential, then $\mathbb{T}^2\big/S$ is homeomorphic to $S^1$.
\end{proposition}

The following version of Brouwers fix point theorem will be useful. 

\begin{lemma}\label{lem:Brouwerrevisited}
	If $U\subseteq \mathbb{T}^2$ is a topological disc and $f\colon
        \mathbb{T}^2\to \mathbb{T}^2$ a homeomorphism such that
        $f(\overline{U})\subseteq U$, then $f$ has a fixed point.
\end{lemma}


\subsection{Monotone factors of toral flows}

We will now show that if the quotient flow under a monotone and upper
semi-continuous equivalence relation is minimal, then either all equivalence
classes are bounded, or all equivalence classes are unbounded.

\begin{lemma} \label{lem:bounded sets form open set}
	Let $S$ be a monotone upper semi-continuous equivalence relation on the
        torus $\mathbb{T}^2$. Then $B:=\{M\subseteq \mathbb{T}^2\big/S;\, M$
        bounded$\}$ is open in $\mathbb{T}^2\big/S$.
\end{lemma}

\begin{proof} 
Let $\pi\colon \mathbb{R}^2\to \mathbb{T}^2$ be the canonical lift of the
torus. For $M\in B$ denote by $C_M$ the set of all connected components of
$\pi^{-1}(M)$. Note that all sets in $C_M$ are bounded and closed, hence compact
in $\mathbb{R}^2$ as $M$ is bounded in $\mathbb{T}^2$ and $S$ is upper
semi-continuous. For $K\in C_M$ we have $C_M=\{K+v;\, v\in \mathbb{Z}^2\}$,
since $M$ is connected in $\mathbb{T}^2$. Thus there is $\epsilon>0$ such that
$B_\epsilon(K)$ and $B_\epsilon(L)$ are disjoint for distinct $K,L\in C_M$.

Let $U:=\pi(B_\epsilon(\pi^{-1}(M)))$ and $\mathcal{V}=\{N\in
\mathbb{T}^2\big/S;\, N\subseteq U\}$. Note that
$\pi^{-1}(U)=B_\epsilon(\pi^{-1}(M))$ which implies that $U$ is open. As $S$ is
upper semi-continuous, we obtain $\bigcup_{V\in \mathcal{V}} V$ to be open in
$\mathbb{T}^2$, i.e.\ $V$ to be open in $\mathbb{T}^2\big/S$ by Lemma
\ref{lem:Uppersemicontinuityproperty_characterization}(ii). Note that $M\in
V$. To show that $V\subseteq B$ let $N\in V$. Then $\bigcup_{L\in C_N}
L=\pi^{-1}(N)\subseteq \pi^{-1}(U)=B_\epsilon(\pi^{-1}(M))=\bigcup_{K\in C_M}
B_\epsilon(K)$. As each $L\in C_N$ is connected there is a unique $K\in C_M$
such that $L\subseteq B_\epsilon(K)$, which implies $L$ to be bounded in
$\mathbb{R}^2$. This proves $N$ to be bounded in $\mathbb{T}^2$.
\end{proof}

\begin{lemma} \label{lem:oneinessentialthenallinessential}
	Let $S$ be a monotone upper semi-continuous $T$-invariant equivalence
        relation on the flow $(\mathbb{T}^2,T)$. If $(\mathbb{T}^2\big/S,T)$ is
        minimal and if there exists an inessential set $M\in \mathbb{T}^2
        \big/$, then all sets in $\mathbb{T}^2\big/S$ are inessential.
\end{lemma}

\begin{proof} Note that a set is inessential if and only if it is bounded.
	As boundedness of subsets of the torus is invariant under homeomorphisms
        from the torus to the torus we know $U:=\{M\subseteq
        \mathbb{T}^2\big/S;\, M$ unbounded$\}$ to be $T$-invariant. Furthermore
        by Lemma \ref{lem:bounded sets form open set} we know $U$ to be closed
        in $\mathbb{T}^2\big/S$. Thus the minimality shows that either
        $U=\emptyset$ or $U=\mathbb{T}^2\big/S$. If we assume the existence of
        an inessential set in $\mathbb{T}^2\big/S$, then $U\neq
        \mathbb{T}^2\big/S$, hence $U=\emptyset$. Thus all elements of
        $\mathbb{T}^2\big/S$ are inessential.
\end{proof}

\begin{theorem} \label{the:characterizationofminimalfactorsofflowsonTorus}
	Let $(\mathbb{T}^2,T)$ be a strongly effective flow. Every monotone
        minimal factor $(\mathbb{T}^2\big/S,T)$ of $(\mathbb{T}^2,T)$ is either
        conjugate to a flow on $\mathbb{T}^2$, conjugate to a flow on the circle
        $S^1$ or conjugate to a flow on a single point.
\end{theorem}

\begin{proof}
	First assume that there exists a doubly essential set $M\in
        \mathbb{T}^2/S$. As no other element of $\mathbb{T}^2\big/S$ can be
        doubly essential we obtain $M$ to be a fix point of $f^t$ for each $t\in
        T$. As $(\mathbb{T}^2\big/S,T)$ is assumed to be minimal, this implies
        that $M=\mathbb{T}^2$, i.e.\ that $\mathbb{T}^2\big/S=\{\mathbb{T}^2\}$
        is the flow on a single point.
	
	As a second case, we assume the existence of an inessential set in
        $\mathbb{T}^2\big/S$. As shown in Lemma
        \ref{lem:oneinessentialthenallinessential}, we obtain all sets in
        $\mathbb{T}^2\big/S$ to be inessential. Let $M\in
        \mathbb{T}^2\big/S$. As $M$ is inessential we know one connected
        component of $\mathbb{T}^2\setminus M$ to be doubly essential and all
        other connected components to be inessential. Let $U$ be an inessential
        connected component of $\mathbb{T}^2\setminus M$. As $U$ is inessential and open
        and $M$ is connected, $U$ is simply connected and thus a topological
        disc. Let $U^*:=\{N\in X\big/S;\, N\subseteq U\}$. As $S$ is monotone,
        we obtain $U=\bigcup_{N\in U^*} N$, hence $U^*\neq \emptyset$. Since $S$
        is upper semi-continuous we have that $U^*$ is open in
        $\mathbb{T}^2\big/S$. Thus the minimality of $(\mathbb{T}^2\big/S,T)$
        implies the existence of $t\in T$ such that $f^t(M)\in U^*$,
        i.e.\ $f^t(M)\subseteq U$. Since $f^t(U)$ is inessential, it is mapped
        to an inessential connected component of $f^t(M)$, hence $f^t(U)\subseteq U$. As
        $f^t(\overline{U})\subseteq f^t(M)\cup f^t(U)\subseteq U$, Lemma
        \ref{lem:Brouwerrevisited} yields the existence of a fix point of $f^t$,
        which contradicts the assumption on $(\mathbb{T}^2,T)$ to be strongly
        effective. Thus there are no inessential connected components of
        $\mathbb{T}^2\setminus M$ and $M$ is non separating. We can now apply
        Moores Theorem (Theorem \ref{the:Moorestheorem}) to obtain a
        homeomorphism between $\mathbb{T}^2\big/S$ and $\mathbb{T}^2$.
	
	For the remaining case assume that all sets in $\mathbb{T}^2\big/S$ are essential. Let $M\in \mathbb{T}^2\big/S$. As each connected component of $\mathbb{T}^2\setminus M$ is a union of essential sets, the complement of $M$ cannot contain any inessential connected components. By connectedness of $M$, $\mathbb{T}^2\setminus M$ cannot contain more than one essential connected component either, so $M$ is non separating. Thus Proposition \ref{pro:kleinerMoore} yields $\mathbb{T}^2\big/S$ to be homeomorphic to $S^1$.
\end{proof}

\begin{remark}
	Note that the argument in the proof of Theorem \ref{the:characterizationofminimalfactorsofflowsonTorus} shows that one can obtain the homeomorphism class of $\mathbb{T}^2\big/S$ by looking at some $M\in \mathbb{T}^2\big/S$. If $M$ is inessential, then $\mathbb{T}^2\big/S$ is homeomorphic to $\mathbb{T}^2$ and all sets in $\mathbb{T}^2\big/S$ are inessential. If $M$ is essential, then we obtain $\mathbb{T}^2\big/S$ to be homeomorphic to $S^1$ and all sets in $\mathbb{T}^2\big/S$ are essential. If $M$ is doubly essential, then $\mathbb{T}^2\big/S$ is homeomorphic to a point and $M=\mathbb{T}^2$.
\end{remark}

If we assume the flow instead of the factor to be minimal we obtain the following. 

\begin{corollary} \label{cor:monotonefactorofminimalflowonntorus}
	Every monotone factor $(\mathbb{T}^2\big/S,T)$ of a minimal flow $(\mathbb{T}^2,T)$ is either conjugate to a flow on $\mathbb{T}^2$, conjugate to a flow on the circle $S^1$ or conjugate to a flow on a single point.
\end{corollary}

\begin{proof}
	Let $G$ be the effective subgroup of $T$ with respect to the flow $(\mathbb{T}^2,T)$. Note that $(\mathbb{T}^2,G)$ is an effective and minimal flow, hence strongly effective. Since $(\mathbb{T}^2\big/S,G)$ is a factor of the minimal flow $(\mathbb{T}^2,G)$ we obtain $(\mathbb{T}^2\big/S,G)$ to be minimal. As the factor map $\pi_S$ is monotone, the statement follows from Theorem \ref{the:characterizationofminimalfactorsofflowsonTorus}.
\end{proof}

Combining Theorem \ref{thm:SeqoncompactandlocallyconnectedHSpaceisMonotone} with Corollary \ref{cor:monotonefactorofminimalflowonntorus} yields the following. 

\begin{corollary}\label{c.mef_on_torus}
	Every maximal equicontinuous factor of a minimal flow on the torus
        $\mathbb{T}^2$ is either conjugate to an equicontinuous flow on
        $\mathbb{T}^2$, to an equicontinuous flow on the circle, or the flow on
        a single point.
\end{corollary}

As another corollary of Theorem \ref{the:characterizationofminimalfactorsofflowsonTorus} we obtain the following.

\begin{corollary} \label{cor:monotoneminimalandeffectivefactorthencharacterisation}
	Every monotone minimal relatively effective factor $(X\big/S,T)$ of a
        flow $(\mathbb{T}^2,T)$ is either conjugate to a flow on $\mathbb{T}^2$,
        conjugate to a flow on the circle or conjugate to a flow on a
        single point.
\end{corollary}

\begin{proof}
	Let $G$ be the effective subgroup of $T$ with respect to
        $(\mathbb{T}^2\big/S,T)$ and note that the induced flows
        $(\mathbb{T}^2,G)$ and $(\mathbb{T}^2\big/S,G)$ are effective and
        $(\mathbb{T}^2\big/S,G)$ is a monotone minimal factor of
        $(\mathbb{T}^2,G)$. As $(\mathbb{T}^2\big/S,G)$ is minimal and
        effective, it is also strongly effective. Thus $(\mathbb{T}^2,G)$ is
        strongly effective by Lemma
        \ref{lem:factorminimalandeffecitveimpliesextensiontobestronglyeffective},
        and Theorem \ref{the:characterizationofminimalfactorsofflowsonTorus}
        implies $\mathbb{T}^2\big/S$ to be homeomorphic either to
        $\mathbb{T}^2$, to $S^1$, or to a single point.
\end{proof}

If the flow is a $\mathbb{Z}$-action, then every minimal factor is effective as shown in the next lemma. 

\begin{lemma}\label{lem:nontrivialmonotoneminimaldiscretefactoriseffective}
	Every non trivial\footnote{A factor $(Y,T)$ of a flow $(X,T)$ is
          considered to be \emph{trivial}, if $Y$ consists of one point.}
        monotone minimal factor of a discrete flow $(\mathbb{T}^2,\mathbb{Z})$
        is effective.
\end{lemma}

\begin{proof}
If $(\mathbb{T}^2\big/S,\mathbb{Z})$ is a non trivial monotone minimal factor of
$(\mathbb{T}^2,\mathbb{Z})$ that is not effective, then there is $n\in
\mathbb{N}$ such that $f^n$ is the identity in $\mathbb{T}^2\big/S$, hence the
orbit of any $M\in X\big/S$ has only finitely many elements. As
$\mathbb{T}^2\big/S$ is assumed to be minimal this implies $\mathbb{T}^2\big/S$
to have finitely many elements. But each of these elements is closed as a subset
of the connected set $\mathbb{T}^2$ and we obtain $\mathbb{T}^2\big/S$ to
consist of exactly one element, a contradiction.
\end{proof}

As the minimal discrete flows on the circle are the irrational circle rotations
we obtain the following corollary from Lemma
\ref{lem:nontrivialmonotoneminimaldiscretefactoriseffective} and Corollary
\ref{cor:monotoneminimalandeffectivefactorthencharacterisation}.
	
\begin{corollary} \label{cor:monotoneMinimalfactorofDiscreteflowonTorus}
	Every monotone minimal factor of a discrete flow
        $(\mathbb{T}^2,\mathbb{Z})$ is either conjugate to a discrete flow on
        $\mathbb{T}^2$, conjugate to an irrational circle rotation or conjugate
        to a discrete flow on a singleton.
\end{corollary}

\begin{remark}
We note that an analogous result can be obtained for homeomorphisms of the
closed annulus. In this case, if the maximal equicontinuous factor is minimal,
then it can only be either an irrational rotation on the circle or the identity
on a singleton.
\end{remark}

Finally, we want to close with two consequences specifically concern the
dynamics of torus homeomorphisms and directly relate to current developments in
this topic (compare
\cite{Jaeger2009Linearisation,JaegerPasseggi2015SemiconjugateToIrrational,%
  KoropeckiPasseggiSambarino2016FMC,JaegerTal2016IrrationalRotationFactors})
. Equivalent conditions for the existence of a semiconjugacy from an
area-preserving torus homeomorphism to an irrational rotation of the two-torus
have been established \cite{Jaeger2009Linearisation}. However, whether the
factor map needs to be monotone in this situation has not been known so far. It
now follows from
Theorem~\ref{thm:SeqoncompactandlocallyconnectedHSpaceisMonotone} and
Corollary~\ref{c.mef_on_torus}.

\begin{corollary}
  Suppose that a homeomorphism $f$ of the two-torus is homotopic to the identity
  and has an irrational translation $\rho$ of the two-torus as a
  factor. Further, assume that the factor map is homotopic to the identity. Then
  $\rho$ is the maximal equicontinuous factor and the factor map is monotone.
\end{corollary}
\proof As $f$ has a two-dimensional irrational rotation $\rho$ as a factor,
Corollary~\ref{c.mef_on_torus} implies that the maximal equicontinuous factor of
$f$ is a two-dimensional irrational rotation $\rho'$ as well. If both rotations
are conjugate, so that $\rho$ is the maximal equicontinuous factor (up to
conjugacy), then monotonicity follows from
Theorem~\ref{thm:SeqoncompactandlocallyconnectedHSpaceisMonotone}.

Hence, suppose for a contradiction that $\rho$ and $\rho'$ are not
conjugate. Suppose that $\rho$ lifts to a translation
$R_\alpha:\R^2\to\R^2,\ x\mapsto x+\alpha$ and $\rho'$ lifts to a corresponding
translation $R_\beta$. Let $h$ be factor map to $\rho$ and $h'$ the maximal
equicontinuous factor map.

The fact that the map $h$ is homotopic to the identity implies that the rotation
number of a suitable lift $F:\R^2\to\R^2$ equals $\alpha$, that is,
$\lim_{n\to\infty} (F^n(x)-x)/n=\alpha$ for all $x\in\R^2$ (e.g. ). Suppose that
the action of $h'$ on the first homology group of $\mathbb{T}^2$, which is
isomorphic to $\Z^2$, is given by the two-by-two integer matrix $A\in
M^{2\times2}(\Z)$. Then a standard argument as in
\cite{kwapisz:1992,Jaeger2009Linearisation} yields that
$\beta=A\alpha$. However, if $A$ is invertible, then this means that $\rho$ and
$\rho'$ are conjugate by the linear torus automorphism $T_A$ induced by $A$. If
$A$ is not invertible, then $\rho'$ is a factor of $\rho$, with factor map
$T_A$, but conversely $\rho$ is not a factor of $\rho'$. This contradicts the
fact that $\rho'$ is the maximal equicontinuous factor.  \qed\medskip

If a torus homeomorphism $f$, homotopic to the identity, is semiconjugate to a
one-dimensional irrational rotation $\rho:\mathbb{T}^1\to\mathbb{T}^1$, then a
construction in \cite{BeguinCrovisierJaeger2016PseudocircleDecomposition} shows
that it is possible that all fibres $h^{-1}(\xi)$ of the factor map
$h:\mathbb{T}^2\to\mathbb{T}^1$ are hereditarily indecomposable continua
(so-called pseudo-circles \cite{bing:1948}) and hence have a `wild' topological
structure. If the maximal equicontinous factor is two-dimensional, then this is
not possible anymore.

Recall that a continuum $K$ is {\em decomposable} if it is the union of two
non-empty subcontinua. We say a vector $u\in\Z^2\setminus\{0\}$ is {\em
  irreducible} if it is not a multiple of some integer vector.
\begin{corollary}
  Suppose that the maximal equicontinuous factor of a torus homeomorphism $f$ is
  (conjugate to) an irrational rotation $\rho$ of the two-torus. Further, assume
  that $h_1$ is a factor map from $f$ to a one-dimensional irrational rotation
  $\rho_1$. Then the fibres of $h_1$ are finite unions of decomposable
  continua. Moreover, if the action $\Z^2\to\Z$ of $h$ on homology is of the
  form $v\mapsto \langle u,v\rangle$ for some irreducible vector
  $u\in\Z^2\setminus\{0\}$, then the fibres of $h$ are decomposable continua.
\end{corollary}
Note that due to the characterisation in
\cite{JaegerPasseggi2015SemiconjugateToIrrational}, the above continua are all
     {\em annular} (the complement of a topological annulus embedded in the
     torus).

\proof Let $h$ be the factor map from $f$ to the maximal equicontinuous factor
$\rho$. Then $h_1=\tilde h\circ h$, where $\tilde h$ is a factor map from $\rho$
to $\rho_1$. However, any factor map from a two-dimensional to a one-dimensional
irrational rotation is of the form $x\mapsto \langle x,u\rangle$ for some
integer vector $u\in\Z^2\setminus \{0\}$. By performing a linear change of
coordinates, we may assume that $u$ is horizontal, that is, $u=(m,0)$ for some
$m\in\N$. This means that $\tilde h$ is just the composition of the projection
to the first coordinate and the covering map $\xi\mapsto m\xi\bmod 1$ on the
circle.

If $m=1$, then it follows from the construction in \cite[Proof of Theorem
  1]{JaegerPasseggi2015SemiconjugateToIrrational} that the fibres of $\tilde
h_1=h\circ h=\pi_1\circ h$ are essential annular continua. However, we have that
\[
   (\pi_1\circ h)^{-1}(\xi) \ = \ h^{-1}(\{\xi\}\times \mathbb{T}^1) \ =
\ h^{-1}([0,1/2]) \cup h^{-1}([1/2,1]) \
\]
and the two sets on the right are subcontinua of the fibre $(\pi_1\circ
h)^{-1}(\xi)$, since preimages of connected sets under monotone maps are are
connected.

If $m>1$, then fibres of $h_1=\tilde h\circ h$ are finite unions of fibres of
$\pi_1\circ h$. This completes the proof.\qed\medskip


 \vspace{10mm} \noindent
\begin{tabular}{l l }
M. Sc. Till Hauser & Prof. Dr. Tobias J\"ager \\ Faculty of Mathematics and
Computer Science & Faculty of Mathematics and Computer Science\\ Institute of
Mathematics & Institute of Mathematics \\ Friedrich Schiller University Jena &
Friedrich Schiller University Jena\\ 07743 Jena & 07743 Jena\\ Germany &
Germany\\{\tt till.hauser@uni-jena.de}&{\tt tobias.jaeger@uni-jena.de}
\end{tabular}

\end{document}